\documentclass[12pt]{article}
\usepackage{amssymb}
\usepackage{amsmath,amsthm}
\usepackage{enumerate}
\usepackage{tikz}

\usepackage[pdftex,
            pdfauthor={Alejandro Estrada-Moreno},
            pdftitle={On the $k$-partition dimension of graphs},
            pdfsubject={Article},
            pdfkeywords={$k$-partition dimension; $k$-metric dimension; partition dimension; metric dimension; graph theory}]{hyperref}
\hypersetup{colorlinks=true}

\hypersetup{colorlinks=true, linkcolor=blue, citecolor=blue,urlcolor=blue}

\newtheorem{remark}{Remark}

\newtheorem{theorem}[remark]{Theorem}
\newtheorem{proposition}[remark]{Proposition}
\newtheorem{corollary}[remark]{Corollary}

\newcommand{\pd}{\operatorname{pd}}
\newcommand{\ter}{\operatorname{ter}}
\providecommand{\keywords}[1]{\noindent\textbf{\textit{Keywords:}} #1}

\title{On the $k$-partition dimension of graphs}

\author{Alejandro Estrada-Moreno
\\
{\small IN3--Computer Science Department,}\\
{\small Open University of Catalonia, Av. Carl Friedrich Gauss 5, 08860 Barcelona, Spain.}
\\{\small aestradamo\@@uoc.edu}
}

\begin{document}
\maketitle

\begin{abstract}
\noindent As a generalization of the concept of the partition dimension of a graph, this article introduces the notion of the $k$-partition dimension. Given a nontrivial connected graph $G=(V,E)$, a partition $\Pi$ of $V$ is said to be a $k$-partition generator of $G$ if any pair of different vertices $u,v\in V$ is distinguished by at least $k$ vertex sets of $\Pi$, \emph{i.e}., there exist at least $k$ vertex sets $S_1,\ldots,S_k\in\Pi$ such that $d(u,S_i)\ne d(v,S_i)$ for every $i\in\{1,\ldots,k\}$. A $k$-partition generator of $G$ with minimum cardinality among all their $k$-partition generators is called a $k$-partition basis of $G$ and its cardinality the $k$-partition dimension of $G$. A nontrivial connected graph $G$ is $k$-partition dimensional if $k$ is the largest integer such that $G$ has a $k$-partition basis. We give a necessary and sufficient condition for a graph to be $r$-partition dimensional and we obtain several results on the $k$-partition dimension for $k\in\{1,\ldots,r\}$.
\end{abstract}


\section{Introduction}
The metric dimension of a metric space was introduced in 1953 by Blumenthal \cite{Blumenthal1953}. However, its study was not considered until 20 years later, when it was applied to the distance between vertices in graphs \cite{Harary1976,Slater1975}. The principal motivation of introducing the concept of metric dimension in graphs was the problem of uniquely determining the location of an intruder in a network. The metric generators were called locating sets in \cite{Slater1975}, while in \cite{Harary1976} they were called resolving sets. The terminology of metric generators for the case of graphs was recently introduced by Seb\"{o} and Tannier in \cite{Sebo2004}, considering that both locating sets and resolving sets are just metric generators for the standard metric defined in graphs.

The literature about the metric generators for graphs shows its highly significant potential to be used for solving a representative number of real life problems, which have been described in several works. For instance, they have been frequently used in graph theory, chemistry, biology, robotics and many other disciplines \cite{Khuller1996,Chartrand2000,Melter1984,Manuel2008}. It was not until 2013 that the theory of metric dimension was developed further for general metric spaces by Bau and Beardon \cite{Sheng2013}.  Again in the context of graph theory, the notion of a $k$-metric dimension was introduced in \cite{Estrada-Moreno2013,Estrada-Moreno2014b,Estrada-Moreno2014,Estrada-Moreno2013corona}, $k$-metric generators and $k$-metric basis, where $k$ is any positive integer and particularly $k=1$ corresponds to the original theory of metric dimension. However, previously this concept had been introduced for the case $k=2$ in \cite{Hernando2008}, where $2$-metric generator of a graph were called fault-tolerant. The idea of the $k$-metric dimension in the context of general metric spaces was studied further in \cite{Beardon-RodriguezVelazquez2016}, although some additional results were also given in the context of graph theory.

The concept of a metric generator of a graph $G$ was not only generalized for the standard metric $d_G(x,y)$, i.e, the length of a shortest path between $x$ and $y$ in $G$, but also for a more general metric defined as $d_{G,t}(x,y)=\min\{d_G(x,y),t\}$, where $t$ is a positive integer. It can be noted that if $t$ is at least the diameter of $G$, then the metric $d_{G,t}$ is equivalent to $d_G$. Articles made in this sense we can mention \cite{Estrada-Moreno2014a,Estrada-Moreno2016k-metric,Estrada-MorenoYeroRodriguez-Velazquez2016, Fernau2018adjacency,Fernau2014Notions} and the Ph.D. thesis \cite{ThesisAlejandro}. Particularly, in these last works the $k$-metric generators of graphs with metric $d_{G,2}$ were called as $k$-adjacency generators.

On the other hand, in order to gain more insight into the metric properties of graphs, several variations of $1$-metric generators have been introduced and studied. Such variations have become more or less known and popular in connection to their applicability or according to the number of challenge that have arisen from them. Among them we could remark  resolving dominating sets \cite{Brigham2003}, weak total resolving sets \cite{Casel2016, Javaid2014Weak}, independent resolving sets \cite{Chartrand2000a}, local metric sets \cite{Barragan-Ramirez2017Simultaneous, Barragan-Ramirez2016local, Okamoto2010}, strong resolving sets \cite{EstrGarcRamRodr2015, Kuziak2013, Oellermann2007}, simultaneous metric dimension \cite{Barragan-Ramirez2017Simultaneous, EstrGarcRamRodr2015, Ramirez_Estrada_Rodriguez_2015, Ramirez-Cruz2018, Ramirez-Cruz-Rodriguez-Velazquez_2014}, strong resolving partitions \cite{GonzalezYero2013, Yero2015} and resolving partitions \cite{Chartrand1998, Chartrand2000b, Fehr2006, Gonzalez2014, RodriguezVelazquez2016Corona, RodriguezVelazquez2014}. A generalization of this last variation will be the object of study of this article.

The paper is structured as follows. In Section \ref{sectPrelim} the main concepts are introduced while Section \ref{sec_k_PartitionDimensional} is devoted to the problem of finding the largest integer $k$ for which there exists a $k$-partition generator of $G$. Particularly, we determine that a graph is $k$-metric dimensional graph if and only if is $k$-partition dimensional, and consequently, we give the previous results for $k$-metric dimensional graphs in terms of the $k$-partition dimensional graphs. Section \ref{sec_k_PartitionDimension} presents some general results related with the computation of $k$-partition dimension as well as some tight bounds on $k$-partition dimension are given. Finally, an upper bound on $k$-partition dimension of trees is presented.

\section{Terminology and basic tools}\label{sectPrelim}

From now on we consider simple and nontrivial connected graphs $G=(V,E)$. It is said that a vertex $v\in V$ \emph{distinguishes} two different vertices $x,y\in V$, if $d_G(v,x)\ne d_G(v,y)$. A set $S\subseteq V$ is said to be a \emph{$k$-metric generator} of $G$ if and only if any pair of different vertices of $G$ is distinguished by at least $k$ elements of $S$. A $k$-metric generator of $G$ with the minimum cardinality among all its $k$-metric generators is a \emph{$k$-metric basis} of $G$. The cardinality of any $k$-metric basis of $G$ is its \emph{$k$-metric dimension}, which will be denoted by $\dim_k(G)$.

Likewise, it is said that a vertex set $S\subseteq V$ distinguishes two different vertices $x,y\in V$, if $d_G(x, S)\ne d_G(y, S)$, where $d_G(a,S)$ is defined as $\displaystyle\min_{v\in S}\{d_G(a,v)\}$. A partition $\Pi$ of $V$ is said to be a \emph{$k$-partition generator} of $G$ if and only if any pair of different vertices of $G$ is distinguished by at least $k$ vertex sets of $\Pi$, {\em i.e.}, for any pair of different vertices $u,v\in V$, there exist at least $k$ vertex sets $S_1,S_2,\ldots,S_k\in \Pi$ such that
\begin{equation}\label{conditionDistinguish}
d_G(u,S_i)\ne d_G(v,S_i),\; \text{for every}\; i\in \{1,\ldots,k\}.
\end{equation}
A $k$-partition generator of the minimum cardinality of G is called a \emph{$k$-partition basis} and its cardinality the \emph{$k$-partition dimension} of $G$, which will be denoted by $\pd_{k}(G)$.

It can be noted that if $k>1$, then every $k$-partition generator $\Pi$ of $G$ is also a $(k-1)$-partition generator. Moreover, $1$-partition generators and $2$-partition generators are resolving partitions and fault-tolerance resolving partitions defined in \cite{Chartrand2000b} and \cite{Salman2009}, respectively. It can be observed that the property of a given partition $\Pi$ of $V$ to be a $k$-partition generator of $G$, where $k\in\{1,2\}$, reduces to checking Condition \eqref{conditionDistinguish} only for those pairs of different vertices $u,v\in V$ in a same vertex set $S_i\in\Pi$. Indeed, $d(u,S_i)=0$ for every vertex $u\in S_i$ but $d(u,S_j)\ne 0$ for $i\ne j$. Hence it follows that if $x\in S_i$ and $y\in S_j$ such that $i\ne j$, then $x,y$ are distinguished by at least $S_i$ and $S_j$. However, if $k\ge 3$ it is necessary to check Condition \eqref{conditionDistinguish} for every pair of different vertices of the graph. In general, for any pair of vertices $u,v\in V, u\ne v$ that belongs to a same vertex set $S_i\in\Pi$ there must be at least other $k$ vertex sets in $\Pi$ distinguishing it. On the other hand, for any pair of vertices $u,v\in V,u\ne v$ such that $u\in S_i$, $v\in S_j$ and $S_i,S_j\in\Pi, S_i\ne S_j$, there must be at least other $k-2$ vertex sets in $\Pi$ distinguishing it.

Throughout the paper, we will use the notation $K_n$, $K_{r,n-r}$, $C_n$ and $P_n$ for complete graphs, complete bipartite graphs, cycle graphs and path graphs of order $n$, respectively.

We use the notation $u\sim v$ if $u$ and $v$ are adjacent vertices and $G \cong H$ if $G$ and $H$ are isomorphic graphs. For a vertex $v$ of a graph $G$, $N_G(v)$ will denote the set of neighbours or \emph{open neighbourhood} of $v$ in $G$, \emph{i.e.} $N_G(v)=\{u \in V(G):\; u \sim v\}$. If it is clear from the context, we will use the notation $N(v)$ instead of $N_G(v)$. The \emph{closed neighbourhood} of $v$ will be denoted by $N[v]=N(v)\cup \{v\}$. Two vertices $x,y$ are called \emph{twins} if $N(x)=N(y)$ or $N[x]=N[y]$. The binary relation ``being twin'' on $V$ is an equivalence relation, and as consequence, it defines the twin equivalence classes.

For the remainder of the paper, definitions will be introduced whenever a concept is needed.

\section{On $k$-partition dimensional graphs}\label{sec_k_PartitionDimensional}

In this section we discuss a natural problem in the study of the $k$-partition dimension of a graph $G$ which consists of finding the largest integer $k$ such that there exists a $k$-partition generator of $G$. We say that a connected graph $G$ is \emph{$k$-partition dimensional} if $k$ is the largest integer such that there exists a $k$-partition basis of $G$.

Given a graph $G$ and two different vertices $x,y\in V(G)$, we denote by $\mathcal{D}_G(x,y)$ the set of vertices that distinguish the pair  $x,y$, \emph{i.e.},

$$\mathcal{D}_G(x,y)=\{z\in V:d(z,x)\ne d(z,y)\}.$$

It can also be noted that for any two different vertices $x,y\in V$ we have that $x,y\in \mathcal{D}_G(x,y)$ and vertices $x,y$ belong to the same twin equivalence class of $G$ if and only if $\mathcal{D}_G(x,y)=\{x,y\}$. With this definition in mind, we show the following proposition which allows us to prove Theorem~\ref{theokPartitionDimensional}, which is the main result in this section.

\begin{proposition}
Let $G=(V,E)$ be a nontrivial connected graph and let $\Pi$ be a partition of $V$. For any vertex set $S\in\Pi$ that distinguishes a pair of different vertices $x,y\in V$, there exists a vertex $z\in S$ that distinguishes $x,y$. 
\end{proposition}

\begin{proof}
Suppose that there exists a $z\in S$ such $d(x,z)=d(x,S)$ and $d(y,z)=d(y,S)$, then since $S$ distinguishes $x,y$ it follows that $d(x,z)=d(x,S)\ne d(y,S)=d(y,z)$, which leads $z$ distinguishes $x,y$. Otherwise, we take a vertex $x'\in S$ such that $d(x,x')=d(x,S)$. If $x'$ distinguishes $x,y$, we are done. Suppose that for any $x'\in S$ such that $d(x,x')=d(x,S)$ we have that $x'$ does not distinguish $x,y$, \textit{i.e.}, $d(x,x')=d(y,x')$. Thus, for any $y'\in S$ such that $d(y,y')=d(y,S)$ it follows that $d(x,y')\ge d(x,S)=d(x,x')=d(y,x')\ge d(y,S)=d(y,y')$. Since $S$ distinguishes $x,y$, we deduce that $d(x,S)>d(y,S)$, and as a consequence, $y'$ distinguishes $x,y$.
\end{proof}

\begin{corollary}\label{remarkDPI}
Let $G=(V,E)$ be a nontrivial connected graph and let $\Pi$ be a $k$-partition generator of $G$. For any pair of different vertices $x,y\in V$ and all vertex sets $S_1,S_2,\ldots,S_r\in \Pi$ that distinguish it, we have that $r\ge k$ and $\mathcal{D}_G(x,y)\cap S_i\ne\emptyset$ for all $i\in\{1,\ldots,r\}$.
\end{corollary}

We now define the following parameter $\mathfrak{d}(G)=\displaystyle\min_{x,y\in V,x\ne y}\{|\mathcal{D}_G(x,y)|\}$.

\begin{theorem}\label{theokPartitionDimensional}
Any graph $G$ of order $n\ge 2$ is $\mathfrak{d}(G)$-partition dimensional and the time complexity of computing $\mathfrak{d}(G)$  is $O(n^3)$.
\end{theorem}

\begin{proof}
Let $G=(V,E)$ be a $k$-partition dimensional graph and let $\Pi$ be a $k$-partition basis of $G$. Let $x,y\in V$ be two different vertices such that $\mathfrak{d}(G)=|\mathcal{D}_G(x,y)|$. By Corollary~\ref{remarkDPI}, for the $r$ vertex sets of $\Pi$ that distinguish $x,y$, we deduce that $k\le r\le\mathfrak{d}(G)$.  
On the other hand, if we take the partition of singletons $\Pi'=\left\{{\left\{{v}\right\}: v \in V}\right\}$ on $V$, then any pair of different vertices $x',y'\in V$ is distinguished by at least $\mathfrak{d}(G)$ elements of $\Pi'$. Thus $\Pi'$ is a $\mathfrak{d}(G)$-partition generator of $G$, and as a consequence, $k\ge \mathfrak{d}(G)$.
Therefore, $k=\mathfrak{d}(G).$

Finally, it was shown in \cite{Yero2017Computing} that  the time complexity of computing $\mathfrak{d}(G)$  is $O(n^3)$.
\end{proof}

As Theorem~\ref{theokPartitionDimensional} shows, in general, the problem of computing
$\mathfrak{d}(G)$ is very easy to solve. Even so, it would be desirable to obtain some general results on this subject. 

Since for every pair of different vertices $x,y\in V$ we have that $|\mathcal{D}(x,y)|\ge 2$, it follows that the partition of singletons $\Pi=\left\{{\left\{{v}\right\}: v \in V}\right\}$ on $V$ is a $2$-partition generator of $G$ and, as a consequence, we deduce that every graph $G$ is $k$-partition dimensional for some $k\ge 2$. On the other hand, for any connected graph $G$ of order $n>2$ there exists at least one vertex $v\in V(G)$ with at least two neighbours. Since $v$ does not distinguish any pair of its different neighbours, $\mathfrak{d}(G)<n$, and by Theorem~\ref{theokPartitionDimensional}, there is no $n$-partition dimensional graph of order $n>2$. Comments above are emphasized  in the next remark.

\begin{remark}\label{remarkKPartitionDimensionaBound}
Let $G$ be a connected graph of order $n\ge 2$. If $n\ge 3$, then $2\le\mathfrak{d}(G)\le n-1$. Moreover, $\mathfrak{d}(G)=n$ if and only if $G\cong K_2$.
\end{remark}

The concept of $k$-metric dimensional graph is closely related to the concept of $k$-partition dimensional graph. A connected graph $G$ is said to be a \emph{$k$-metric dimensional graph}, if $k$ is the largest integer such that there exists a $k$-metric basis of $G$. The following theorem, which was previously stated in \cite{Estrada-Moreno2013}, will allow us to establish an equivalence between both concepts. 

\begin{theorem}\label{theokmetricdimensional}{\rm \cite{Estrada-Moreno2013}}
A connected graph  $G$ is $k$-metric dimensional if and only if $k=\mathfrak{d}(G)$.
\end{theorem}

Therefore, by Theorems \ref{theokPartitionDimensional} and \ref{theokmetricdimensional} we have the following result.

\begin{theorem}\label{theoRelationPartitionMetricDimensional}
A connected graph  $G$ is $k$-partition dimensional if and only if it is $k$-metric dimensional.
\end{theorem}

The following results that we present (from Proposition \ref{coro2DimensionalG} to Theorem \ref{theoTreeDimensionalK}) were previously obtained in \cite{Estrada-Moreno2013} for the $k$-metric dimensional graphs. By Theorem~\ref{theoRelationPartitionMetricDimensional} all these results are equivalent for $k$-partition dimensional graphs, therefore for the sake of completeness we rewrite them in terms of $k$-partition dimensional graphs. 

\begin{proposition}\label{coro2DimensionalG}
A connected graph $G$ satisfies that $\mathfrak{d}(G)=2$ if and only if there are at least two vertices of $G$ belonging to the same twin equivalence class.
\end{proposition}

\begin{theorem}
A connected graph $G$ of order $n\ge 3$ holds that $\mathfrak{d}(G)=n-1$ if and only if $G$ is a path or $G$ is an odd cycle.
\end{theorem}

\begin{proposition}\label{propKClyce}
If $C_{n}$ is a even cycle of order $n$, then $\mathfrak{d}(C_n)=n-2$.
\end{proposition}

The \emph{Cartesian product graph} $G\square H$, of two graphs $G=(V_{1},E_{1})$ and $H=(V_{2},E_{2})$, is the graph whose vertex set is $V(G\square H)=V_{1}\times V_{2}$ and any two distinct vertices $(x_{1},x_{2}),(y_{1},y_{2})\in V_{1}\times V_{2}$ are adjacent in $G\square H$ if and only if either:

\begin{enumerate}[(a)]
\item\label{cartesian1}$\ $ $x_{1}=y_{1}$ and $x_{2}\sim y_{2}$, or
\item\label{cartesian2}$\ $ $x_{1}\sim y_{1}$ and $x_{2}=y_{2}$.
\end{enumerate}

\begin{proposition}\label{corolGCartesian}
If $G$ and $H$ are two connected graphs of order $n\ge 2$ and $n'\ge 3$, respectively, then $\mathfrak{d}(G\square H)\geq 3$.
\end{proposition}

A {\em clique} in a graph $G$ is a set of vertices $S$ such that the subgraph induced by $S$ is isomorphic to a complete graph. The maximum cardinality of a clique in a graph $G$ is the {\em clique number} and it is denoted by $\omega(G)$.

\begin{theorem}\label{clique-versus-k}
If $G$ is a graph of order $n$ different from a complete graph, then $\mathfrak{d}(G)\le n-\omega(G)+1$.
\end{theorem}

In order to continue presenting our next results, we need to introduce some definitions. A vertex of degree at least three in a graph $G$ will be called a \emph{major vertex} of $G$. Any end vertex $u$ of $G$ is said to be a \emph{terminal vertex} of a major vertex $v$ of $G$ if $d_{G}(u,v)<d_{G}(u,w)$ for every other major vertex $w$ of $G$. The \emph{terminal degree} $\ter(v)$ of a major vertex $v$ is the number of terminal vertices of $v$. A major vertex $v$ of $G$ is an \emph{exterior major vertex} of $G$ if it has positive terminal degree. Let $\mathcal{M}(G)$\label{calM} be the set of exterior major vertices of $G$ having terminal degree greater than one.

From now on we consider that $w\in \mathcal{M}(G)$. Given a terminal vertex $u$ of $w$, we denote by  $P(u,w)$ the shortest path that starts at $u$ and ends at $w$. Let $l(u,w)$ be the length of $P(u,w)$. Now, given two terminal vertices $u,v$ of $w$ we denote by  $P(u,v)$ the shortest path from $u$ to $v$ containing $w$, and by $\varsigma(u,v)$ the length of $P(u,w,v)$. Notice that, by definition of exterior major vertex, $P(u,w,v)$ is obtained by concatenating the paths $P(u,w)$ and $P(w,v)$, where $w$ is the only vertex of degree greater than two lying on these paths. We also define $U(w)$ as the set of terminal vertices of $w$. Thus, it follows that $\ter(w)=|U(w)|$. Finally, we define $\varsigma(w)=\displaystyle\min_{u,v\in U(w),u\ne v}\{\varsigma(u,v)\}$ and $l(w)=\displaystyle\min_{z\in U(w)}\{l(z,w)\}$. From the local parameters above we define the following global parameter 
$$\varsigma(G)=\min_{w\in \mathcal{M}(G)}\{\varsigma(w)\}.$$\label{oddDefinition}

An example which helps to understand the notation above is given in Figure~\ref{example-G}.

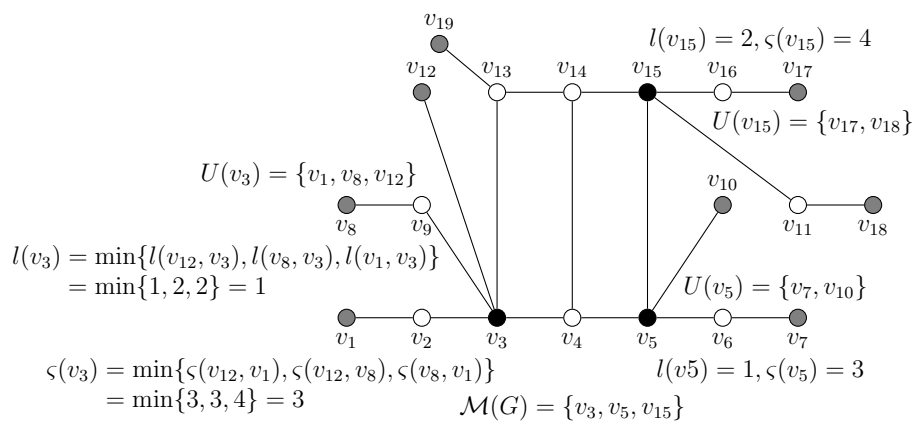
\begin{figure}[!ht]
\centering
\begin{tikzpicture}[transform shape, inner sep = .8mm]
\node [draw=black, shape=circle, fill=gray,text=black] (v1) at (0,0) {};
\node [scale=.8] at ([yshift=-.3 cm]v1) {$v_1$};
\node [draw=black, shape=circle, fill=gray,text=black] (v8) at (0,1.5) {};
\node [scale=.8] at ([yshift=-.3 cm]v8) {$v_8$};
\node [draw=black, shape=circle, fill=gray,text=black] (v12) at (1,3) {};
\node [scale=.8] at ([yshift=.3 cm]v12) {$v_{12}$};
\node [draw=black, shape=circle, fill=white,text=black] (v2) at ([shift=({1,0})]v1) {};
\node [scale=.8] at ([yshift=-.3 cm]v2) {$v_2$};
\node [draw=black, shape=circle, fill=black,text=black] (v3) at ([shift=({1,0})]v2) {};
\node [scale=.8] at ([yshift=-.3 cm]v3) {$v_3$};
\node [draw=black, shape=circle, fill=white,text=black] (v4) at ([shift=({1,0})]v3) {};
\node [scale=.8] at ([yshift=-.3 cm]v4) {$v_4$};
\node [draw=black, shape=circle, fill=black,text=black] (v5) at ([shift=({1,0})]v4) {};
\node [scale=.8] at ([yshift=-.3 cm]v5) {$v_5$};
\node [draw=black, shape=circle, fill=white,text=black] (v6) at ([shift=({1,0})]v5) {};
\node [scale=.8] at ([yshift=-.3 cm]v6) {$v_6$};
\node [draw=black, shape=circle, fill=gray,text=black] (v7) at ([shift=({1,0})]v6) {};
\node [scale=.8] at ([yshift=-.3 cm]v7) {$v_7$};
\node [draw=black, shape=circle, fill=white,text=black] (v9) at ([shift=({1,0})]v8) {};
\node [scale=.8] at ([yshift=-.3 cm]v9) {$v_9$};
\node [draw=black, shape=circle, fill=gray,text=black] (v10) at ([shift=({4,0})]v9) {};
\node [scale=.8] at ([yshift=.3 cm]v10) {$v_{10}$};
\node [draw=black, shape=circle, fill=white,text=black] (v11) at ([shift=({1,0})]v10) {};
\node [scale=.8] at ([yshift=-.3 cm]v11) {$v_{11}$};
\node [draw=black, shape=circle, fill=gray,text=black] (v18) at ([shift=({1,0})]v11) {};
\node [scale=.8] at ([yshift=-.3 cm]v18) {$v_{18}$};
\node [draw=black, shape=circle, fill=white,text=black] (v13) at ([shift=({1,0})]v12) {};
\node [scale=.8] at ([yshift=.3 cm]v13) {$v_{13}$};
\node [draw=black, shape=circle, fill=white,text=black] (v14) at ([shift=({1,0})]v13) {};
\node [scale=.8] at ([yshift=.3 cm]v14) {$v_{14}$};
\node [draw=black, shape=circle, fill=black,text=black] (v15) at ([shift=({1,0})]v14) {};
\node [scale=.8] at ([yshift=.3 cm]v15) {$v_{15}$};
\node [draw=black, shape=circle, fill=white,text=black] (v16) at ([shift=({1,0})]v15) {};
\node [scale=.8] at ([yshift=.3 cm]v16) {$v_{16}$};
\node [draw=black, shape=circle, fill=gray,text=black] (v17) at ([shift=({1,0})]v16) {};
\node [scale=.8] at ([yshift=.3 cm]v17) {$v_{17}$};
\node [draw=black, shape=circle, fill=gray,text=black] (v19) at ([shift=(140:1cm)]v13) {};
\node [scale=.8] at ([yshift=.3 cm]v19) {$v_{19}$};

\node [scale=.8] at ([yshift=-1.2cm]v4) {$\mathcal{M}(G)=\{v_3,v_5,v_{15}\}$};
\node [scale=.8] at ([shift={(-.5cm,.4cm)}]v8) {$U(v_3)=\{v_1,v_8,v_{12}\}$};
\node [scale=.8] at ([shift={(-.3cm,.4cm)}]v7) {$U(v_5)=\{v_7,v_{10}\}$};
\node [scale=.8] at ([shift={(.2cm,-.4cm)}]v17) {$U(v_{15})=\{v_{17},v_{18}\}$};
\node [scale=.8] at ([shift={(-3cm,-.7cm)}]v3) {$\varsigma(v_3)=\displaystyle\min\{\varsigma(v_{12},v_1),\varsigma(v_{12},v_8),\varsigma(v_8,v_1)\}$};
\node [scale=.8] at ([shift={(-3.88cm,-1.1cm)}]v3) {$=\min\{3,3,4\}=3$};
\node [scale=.8] at ([shift={(-1.6cm,.8cm)}]v1) {$l(v_3)=\min\{l(v_{12},v_3),l(v_8,v_3),l(v_1,v_3)\}$};
\node [scale=.8] at ([shift={(-2.38cm,.4cm)}]v1) {$=\min\{1,2,2\}=1$};
\node [scale=.8] at ([shift={(1.5cm,-.7cm)}]v5) {$l(v5)=1, \varsigma(v_5)=3$};
\node [scale=.8] at ([shift={(1.5cm,.7cm)}]v15) {$l(v_{15})=2, \varsigma(v_{15})=4$};

\draw[black] (v1) -- (v2)  -- (v3) -- (v4) -- (v5) -- (v6) -- (v7);
\draw[black] (v8) -- (v9)  -- (v3) -- (v12);
\draw[black] (v3) -- (v13)  -- (v14) -- (v15) -- (v16) -- (v17);
\draw[black] (v11) -- (v15)  -- (v5) -- (v10);
\draw[black] (v4) -- (v14);
\draw[black] (v11) -- (v18);
\draw[black] (v13) -- (v19);
\end{tikzpicture}
\caption{A graph $G$ with $\varsigma(G)=3$. The vertices belonging to $\mathcal{M}(G)$ are depicted in black while the terminal vertices are represented in gray.}
\label{example-G}
\end{figure}

According to this notation we present the following result.

\begin{theorem}\label{coroKcota}
If $G$ is a connected graph such that  $\mathcal{M}(G)\ne \emptyset$, then $\mathfrak{d}(G)\leq \varsigma(G)$.
\end{theorem}

The upper bound on $\mathfrak{d}(G)$ given in Theorem \ref{coroKcota} is sharp for the case of trees different from paths, as the following result shows.

\begin{theorem}\label{theoTreeDimensionalK}
If $T$ is a tree different from a path, then $\mathfrak{d}(T)=\varsigma(T)$.
\end{theorem}

For a tree $T$ different from a path, the value of parameter $\varsigma(T)$ can be computed in linear time with respect to the order of $T$, as it was shown in \cite{Yero2017Computing}.

\section{On the $k$-partition dimension of a graph}\label{sec_k_PartitionDimension}
In this section we study the problem of computing or bounding the $k$-partition dimension of a graph. Since for any connected graph $G=(V,E)$ of order $n\ge 2$ and any $k$-partition basis $\Pi$ of $G$, we have all pairs of different vertices of $V$ are distinguished by at least $k$ vertex set of $\Pi$, it follows that $\pd_k(G)\ge k$ for any $k>1$. Particularly, $2\le\pd_1(G)\le n$. On the other hand, by definition we know that every $k$-partition basis of $G$, for $k>1$, is a $(k-1)$-partition generator of $G$. Therefore, the next result follows.

\begin{theorem}\label{trivialInequality}
Let G be a nontrivial connected graph and let $k_1,k_2$ be two integers. If $1<k_1<k_2\le\mathfrak{d}(G)$, then $k_1\le\pd_{k_1}(G)\le\pd_{k_2}(G)\le n$. Moreover, $2\le\pd_1(G)\le n$.
\end{theorem}

It was shown in \cite{Chartrand2000b} that $\pd_1(G)=2$ if and only if $G\cong P_n$, and $\pd_1(G)=n$ if and only if $G\cong K_n$. We consider the limit case of the trivial bound $\pd_{k}(G)\ge k$ for $k\ge 2$.

\begin{theorem}\label{pd_k=k}
If $G$ is a nontrivial connected graph, then $\pd_{k}(G)=k$ if and only if $k=2$ and $G\cong K_{2}$.
\end{theorem}

\begin{proof}
It is readily seen that if $k=2$ and $G\cong K_{2}$, then the partition of singletons on vertex set of $G$ is a $2$-partition basis of $G$ with cardinality $2$.

Conversely, assume that $\pd_{k}(G)=k$. By Thereom \ref{trivialInequality}, we have that $k\ge 2$. Let $\Pi$ be a $k$-partition basis of $G=(V,E)$. If there exists one vertex set $S\in\Pi$ such that $|S|\ge 2$, then for any pair of different vertices $x,y\in S$ we need at least other $k$ vertex set in $\Pi$, which contradicts the fact that $|\Pi|=k$. Thus, $\Pi$ is the  partition of singletons on vertex set of $G$, which leads to $\pd_{n}(G)=n$. By Remark~\ref{remarkKPartitionDimensionaBound}, we conclude that $k=2$ and $G\cong K_{2}$.
\end{proof}

The existing relationship between $k$-metric dimension and $k$-partition dimension of graph for $k=1$ was shown in \cite{Chartrand2000b}.

\begin{theorem}\label{theoRelationParDim_1}{\rm \cite{Chartrand2000b}}
If $G$ is a nontrivial connected graph, then $$\pd_1(G)\le\dim_1(G)+1$$
\end{theorem}

This upper bound, although in general is true, is never achieved for $k\ge 2$ in some cases as we will see below.

\begin{theorem}\label{theoRelationParDim_k}
If $G=(V,E)$ is a connected graph of order $n\ge 2$, then for any $k\in\{1,2,\ldots,\mathfrak{d}(G)-1\}$ 
$$\pd_k(G)\le\dim_k(G)+1.$$ Moreover, if $\dim_{\mathfrak{d}(G)}(G)<n$, then $\pd_{\mathfrak{d}(G)}(G)\le\dim_{\mathfrak{d}(G)}(G)+1$, otherwise, $\pd_{\mathfrak{d}(G)}(G)\le\dim_{\mathfrak{d}(G)}(G)$.
\end{theorem}

\begin{proof}
Let $\dim_k(G)=s$ and let $W=\{w_1,w_2,\ldots,w_s\}$ be a $k$-metric basis of $G$. It was shown in \cite{Estrada-Moreno2013} that for $k\in\{1,2,\ldots,\mathfrak{d}(G)-1\}$ we have that $\dim_k(G)<n$. Firstly, we assume that either $k\in\{1,2,\ldots,\mathfrak{d}(G)-1\}$ or $k=\mathfrak{d}(G)$ and $\dim_{\mathfrak{d}(G)}(G)<n$, and as a consequence, $W\ne V$. We consider the partition $\Pi=\{S_1,S_2,\ldots,S_{s+1}\}$ of $V$, where $S_i=\{w_i\}$ for all $i\in\{1,\ldots,s\}$ and $S_{s+1}=V-W\ne\emptyset$. Since $W$ is a $k$-metric basis of $G$ all pairs of different vertices are distinguished by at least $k$ vertex sets of $\Pi-S_{s+1}$. Therefore, $\Pi$ is a $k$-partition generator of $G$, and as a consequence, $\pd_k(G)\le|\Pi|=\dim_k(G)+1$. 

Suppose that $k=\mathfrak{d}(G)$ and $\dim_{\mathfrak{d}(G)}(G)=n$. In this case we take the partition of singletons $\Pi'=\{S_1,S_2,\ldots,S_n\}$ on $V$. Since $W$ is a $k$-metric basis of $G$ all pairs of different vertices are distinguished by at least $k$ vertex sets of $\Pi'$. Therefore, $\Pi'$ is a $k$-partition generator of $G$, and as a consequence, $\pd_{\mathfrak{d}(G)}(G)\le|\Pi'|=\dim_{\mathfrak{d}(G)}(G)$.
\end{proof}

It was shown in \cite{Chartrand2000b} that the upper bound in Theorem~\ref{theoRelationParDim_k} is sharp for the graphs $P_n$, $C_n$, $K_n$ and $K_{1,n}$ when $k=1$. We will show other graphs that achieve the upper bound of Theorem~\ref{theoRelationParDim_k} for $k>1$. To this end, we present some results obtained previously.

\begin{theorem}\label{TrivialUpperBound}{\rm \cite{Estrada-Moreno2016k-metric}}
For any connected graph $G$ of order $n$ and any $k\in \{1,\dots,\mathfrak{d}(G)\}$,
$$\dim_k(G)\le n-\mathfrak{d}(G)+k.$$
\end{theorem}

By Theorems \ref{theoRelationParDim_k} and \ref{TrivialUpperBound} we deduce the next result.

\begin{proposition}
If $G$ is a connected graph of order $n\ge 2$, then for any $k\in\{1,\ldots,\mathfrak{d}(G)-1\}$
$$\pd_k(G)\le n-\mathfrak{d}(G)+k + 1.$$
Moreover, if $\pd_k(G)=n$, then $k\in\{\mathfrak{d}(G)-1,\mathfrak{d}(G)\}$. 
\end{proposition}

Given a connected graph $G=(V,E)$, we define the following parameter $$\mathfrak{d}^*(G)=\displaystyle\max_{x,y\in V,x\ne y}\{|\mathcal{D}_G(x,y)|\}.$$ It can be proved that the time complexity of computing $\mathfrak{d}^*(G)$ is $O(n^3)$ applying a procedure similar to that used in \cite{Yero2017Computing} to prove the time complexity of computing $\mathfrak{d}(G)$.

\begin{theorem}\label{necessaryConditionPD=n}
For any connected graph $G$ of order $n\ge 2$ and any $k\in \{1,\dots,\mathfrak{d}(G)\}$ such that $\mathfrak{d}^*(G)\le k+1$, we have that $\pd_k(G)=n$. Moreover, if $k\in\{1,2\}$, then $\pd_k(G)=n$ if and only if $\mathfrak{d}^*(G)\le k+1$.
\end{theorem}

\begin{proof}
Since $k\le\mathfrak{d}(G)\le \mathfrak{d}^*(G)\le k+1$ it follows that $k\in\{\mathfrak{d}(G)-1, \mathfrak{d}(G)\}$. On the other hand, considering the partition of singletons on $V$ is a $k$-partition generator of $G$, it follows that $\pd_k(G)\le n$. Finally, since $\mathfrak{d}^*(G)\le k+1$, any partition $\Pi$ on $V$, different from the partition of singletons on $V$, has at least one element with two vertices that are distinguished by at most $k-1$ elements of $\Pi$. Therefore, $\pd_k(G)\ge n$, which leads to $\pd_k(G)=n$.

Suppose that $k\in\{1,2\}$ and $\pd_k(G)=n$. Assume for the purpose of contradiction that $\mathfrak{d}^*(G)\ge k+2$. Let $x,y$ be two vertices of $G$ such that $|\mathcal{D}_G(x,y)|=\mathfrak{d}^*(G)$. Since for any $k\in\{1,2\}$ we have that any partition $\Pi$ of $V$ is a $k$-partition generator of $G$ if every pair of different vertices belonging to a same vertex set of $\Pi$ is distinguished by other $k$ vertex sets of $\Pi$. Thus, if $\mathfrak{d}^*(G)\ge k+2$, then the partition $\displaystyle\Pi'=\bigcup_{z\in V-\{x,y\}}\{z\}\cup\{x,y\}$ is $k$-partition generator of $G$, which is a contradiction. Therefore, $\mathfrak{d}^*(G)\le k+1$ for $k\in\{1,2\}$.
\end{proof}

\begin{corollary}\label{necessaryConditionPD=nCor}
If $G$ is a nontrivial connected graph such that $\mathfrak{d}(G)=\mathfrak{d}^*(G)$, then $\pd_k(G)=n$ for any $k\in\{\mathfrak{d}(G)-1, \mathfrak{d}(G)\}$.
\end{corollary}

Examples of graphs that satisfy Corollary~\ref{necessaryConditionPD=nCor} are $K_n$, $C_{2k+1}$, where $k$ is a positive integer, and the wheel graph $W_{1,5}$. It can be noted that $\mathfrak{d}(K_n)=\mathfrak{d}^*(K_n)=2$, $\mathfrak{d}(C_{2k+1})=\mathfrak{d}^*(C_{2k+1})=2k$, and finally, $\mathfrak{d}(W_{1,5})=\mathfrak{d}^*(W_{1,5})=4$. Particularly, graphs that hold Theorem~\ref{necessaryConditionPD=n} such that $\mathfrak{d}(G)\ne\mathfrak{d}^*(G)$, we have $K_n-e$ for $k=2$, paths $P_n$ ($n\ge 3$) for $k=n-1$ and finally, the fan graph $F_{1,4}$ for $k=3$. We recall that the \emph{wheel graph} $W_{1,n}$ is equal to $K_1+C_n$ and the \emph{fan graph} $F_{1,n}$ is equal to $K_1+P_n$, where $G+H$ represents the \emph{join graph} between the graphs $G=(V_{1},E_{1})$ and $H=(V_{2},E_{2})$ whose vertex set is $V(G+H)=V_{1}\cup V_{2}$ and edge set is $E(G+H)=E_{1}\cup E_{2}\cup \{\{u,v\}\,:\,u\in V_{1},v\in V_{2}\}$.

The following two results were obtained previously, but the proofs, especially for Proposition~\ref{PropAllVertices2}, are reduced by applying the new approach presented in Theorem~\ref{necessaryConditionPD=n}.

\begin{proposition}{\rm \cite{Chartrand2000b}}
If $G$ is a connected graph of order $n\ge 2$, then $\pd_1(G)=n$ if and only if $G\cong K_n$. 
\end{proposition}

\begin{proposition}{\rm \cite{ChaudhryJavaidSalman2010}}\label{PropAllVertices2}
If $G$ is a connected graph of order $n\ge 2$, then $\pd_2(G)=n$ if and only if $G\cong K_n$ or $G\cong K_n-e$. 
\end{proposition}

\begin{proof}
Suppose that $\pd_1(G)=n$. By Theorem~\ref{necessaryConditionPD=n}, $\pd_1(G)=n$ if and only if  $\mathfrak{d}^*(G)\le 2$. Since $2\le\mathfrak{d}(G)\le\mathfrak{d}^*(G)\le 2$, it follows that $\mathfrak{d}(G)=\mathfrak{d}^*(G)=2$. It is readily seen that $\mathfrak{d}^*(G)=2$ if and only if $G\cong K_n$.

Assume that $\pd_2(G)=n$. By Theorem~\ref{necessaryConditionPD=n}, $\pd_2(G)=n$ if and only if  $\mathfrak{d}^*(G)\le 3$. Thus, either $\mathfrak{d}^*(G)=2$ or $\mathfrak{d}^*(G)=3$. Since $\mathfrak{d}^*(G)=2$ if and only if $G\cong K_n$, from now on we assume that $\mathfrak{d}^*(G)=3$. Hence, the diameter $d$ of $G$ is at least two. Suppose for the purpose of contradiction that $d\ge 3$, and let $v_0,v_1,\ldots,v_d$ a path in $G$ of length $d$ such that $v_0$ and $v_d$ are diametrical vertices. Since $\mathfrak{d}^*(G)\ge |\mathcal{D}_G(v_0,v_1)|\ge |\{v_0,v_1,\ldots,v_d\}|\ge 4$, this leads to a contradiction. Hence $d=2$. Let $x,y$ be two vertices of $G$ such that $d(x,y)=2$. Suppose that there exists another vertex $y'$ such that $y'\ne y$ and $d(x,y')=2$. If there exists a vertex $z$ such that $d(x,z)=d(z,y)=d(z,y')=1$, then $|\mathcal{D}_G(x,z)|\ge |\{x,z,y,y'\}|\ge 4$, which is another contradiction. Thus, for every vertex $z$ such that $d(x,z)=d(z,y)=1$ we have that $d(z,y')=2$. Analogously, for every  $z'$ such that $d(x,z')=d(z',y')=1$ we have that $d(z',y)=2$. Let $z,z'$ be two vertices of $G$ such that $d(x,z)=d(z,y)=1$, $d(z,y')=2$, $d(x,z')=d(z',y')=1$ and $d(z',y)=2$. In this case $|\mathcal{D}_G(y,y')|\ge |\{y,y',z,z'\}|\ge 4$, which is a contradiction again. Thus, for any diametrical vertex, there exists only one vertex which is at distance two from it. Suppose that in addition to $x,y$ there exist other two diametrical vertices $x',y'$. Then $d(x,x')=d(x,y')=1$ and $d(y,x')=d(y,y')=1$. In this case $|\mathcal{D}_G(x,x')|\ge |\{x,x',y,y'\}|\ge 4$, which is also a contradiction. Therefore, $G$ is a graph of diameter $2$ that has only one pair of diametrical vertices, i.e., $G\cong K_n-e$. It is straightforward to check that $\mathfrak{d}^*(K_n-e)=3$, and with this the proof is completed.
\end{proof}

\subsection{The particular case of trees}

In this subsection we focus our study in $k$-partition dimension of a tree. Firstly, we study the paths  which are the simplest trees. 

\begin{proposition}\label{pathOrderAtLeast3}
If $P_n$ is a path of order $n\ge 3$, then $\pd_k(P_n)=k+1$ for any $k\in\{1,2,\ldots,n-1\}$.
\end{proposition}

\begin{proof}
By Theorem~\ref{pd_k=k}, we have that $\pd_k(P_n)\ge k+1$. On the other hand, let $q$ and $r$ be the integers such that $n=(k+1)q+r$ and $0\le r\le k$. Namely, $q$ is the quotient and $r$ is the remainder when we divide $n$ by $k+1$. Let $\{v_1,v_2,\ldots,v_n\}$ be the vertex set of $P_n$ such that $v_i\sim v_{i+1}$ for $i\in\{1,2,\ldots,n-1\}$. For any $j\in\{1,\ldots,r\}$ we define $S_j=\{v_{(j-1)(q+1)+1},v_{(j-1)(q+1)+2},\ldots,v_{j(q+1)}\}$ and for any $j\in\{r+1,\ldots,k+1\}$ we define $S_j=\{v_{r(q+1)+\left(j-(r+1)\right)q+1},v_{r(q+1)+\left(j-(r+1)\right)q+2},\ldots,$ $v_{r(q+1)+\left(j-r\right)q}\}$. We take the set $\Pi=\{S_1,S_2,\ldots,S_{k+1}\}$ and we claim that $\Pi$ is a $k$-partition generator of $P_n$. Let $x,y$ be two vertex of $P_n$. Suppose that $x,y\in S_i$ for some $i\in\{1,\ldots,k+1\}$. By the construction of $\Pi$, we deduce that $x,y$ are distinguished by the $k$ vertex sets in $\Pi-\{S_i\}$. Suppose now that $x\in S_i$ and $y\in S_j$ such that $i\ne j$ and $i,j\in\{1,\ldots,k+1\}$. In this case there exists at most one vertex set in $\Pi-\{S_i,S_j\}$ that does not distinguish $x,y$, and as a consequence, $x,y$ are distinguished by at least $k$ vertex sets of $\Pi$. Therefore, $\Pi$ is a $k$-partition generator of $G$ which leads to $\pd_k(P_n)\le|\Pi|=k+1$.
\end{proof}

By Theorem~\ref{pd_k=k} and Proposition~\ref{pathOrderAtLeast3} along with the fact that $\pd_1(P_2)=2$, we know the $k$-partition dimension of any path $P_n$. We point out that the formula $\pd_1(P_n)=2$ was obtained previously in \cite{Chartrand2000b}.

From now on we propose an upper bound on the $k$-partition dimension of trees different from paths. Firstly, we recall that an upper bound on $k$-partition dimension of a graph was given by Theorem~\ref{theoRelationParDim_k} in terms of its $k$-metric dimension. A formula for the $k$-metric dimension of trees that are not paths has been established in \cite{Estrada-Moreno2013}. In order to present this formula, we need an additional definition along with those that were already used for the definition of the parameter $\varsigma(G)$ on page~\pageref{oddDefinition}. For any exterior major vertex $w$ that belongs the vertex set of a tree $T$ and having terminal degree greater than one, {\it i.e.}, $w\in \mathcal{M}(T)$, we define the following function for any $k\in\{1,\ldots,\varsigma(T)\}$,
\[
I_k(w)=\left\{ \begin{array}{ll}
\left(\ter(w)-1\right)\left(k-l(w)\right)+l(w), & \text{if } l(w)\le\left\lfloor\frac{k}{2}\right\rfloor,\\
& \\
\left(\ter(w)-1\right)\left\lceil\frac{k}{2}\right\rceil+\left\lfloor\frac{k}{2}\right\rfloor, & \text{otherwise.}
\end{array}
\right.
\]

We can now state the formula for the $k$-metric dimension of any tree different from path.

\begin{theorem}{\rm\cite{Estrada-Moreno2013}}\label{theoTreeDimK}
If $T$ is a tree that is not a path, then for any $k\in \{1,\ldots,\varsigma(T)\}$, $$\dim_{k}(T)=\sum_{w\in \mathcal{M}(T)}I_{k}(w).$$
\end{theorem}

It was shown in \cite{Yero2017Computing} that $\dim_{k}(T)$ can be computed in linear time with respect to the order of any tree $T$ different from path. On the other hand, by Theorem \ref{theoTreeDimK} it can be noted that $\dim_k(T)<n$. Therefore, according to Theorem~\ref{theoRelationParDim_k}, and considering that $\mathfrak{d}(T)=\varsigma(T)$ by Theorem~\ref{theoTreeDimensionalK}, we deduce the next upper bound on $k$-partition dimension of $T$ for any $k\in \{1,\ldots,\varsigma(T)\}$

\begin{equation}\label{trivialUpperBoundTree}
\pd_k(T)\le\dim_{k}(T)+1=\sum_{w\in \mathcal{M}(T)}I_{k}(w)+1.
\end{equation}

The previous bound is tight and it is achieved for the tree $T$ shown in Figure~\ref{exampleAchievesUpperBound}.

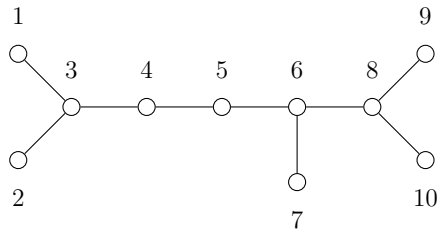
\begin{figure}[!ht]
\centering
\begin{tikzpicture}[transform shape, inner sep = .8mm]
\node [draw=black, shape=circle, fill=white] (6) at (0,0) {};
\node [scale=.8] at ([yshift=.5 cm]6) {$6$};
\node [draw=black, shape=circle, fill=white] (5) at ([xshift=-1 cm]6) {};
\node [scale=.8] at ([yshift=.5 cm]5) {$5$};
\node [draw=black, shape=circle, fill=white] (8) at ([xshift=1 cm]6) {};
\node [scale=.8] at ([yshift=.5 cm]8) {$8$};
\node [draw=black, shape=circle, fill=white] (7) at ([yshift=-1 cm]6) {};
\node [scale=.8] at ([yshift=-.5 cm]7) {$7$};
\node [draw=black, shape=circle, fill=white] (4) at ([xshift=-2 cm]6) {};
\node [scale=.8] at ([yshift=.5 cm]4) {$4$};
\node [draw=black, shape=circle, fill=white] (3) at ([xshift=-3 cm]6) {};
\node [scale=.8] at ([yshift=.5 cm]3) {$3$};
\node [draw=black, shape=circle, fill=white] (1) at ([shift=(135:1cm)]3) {};
\node [scale=.8] at ([yshift=.5 cm]1) {$1$};
\node [draw=black, shape=circle, fill=white] (2) at ([shift=(225:1cm)]3) {};
\node [scale=.8] at ([yshift=-.5 cm]2) {$2$};
\node [draw=black, shape=circle, fill=white] (9) at ([shift=(45:1cm)]8) {};
\node [scale=.8] at ([yshift=.5 cm]9) {$9$};
\node [draw=black, shape=circle, fill=white] (10) at ([shift=(-45:1cm)]8) {};
\node [scale=.8] at ([yshift=-.5 cm]10) {$10$};
\draw (1)--(3)--(4)--(5)--(6)--(8)--(9);
\draw (2)--(3);
\draw (6)--(7);
\draw (8)--(10);
\end{tikzpicture}
\caption{The set $\Pi=\left\{\{1\},\{2\}\{3,4,5,6,7,8\},\{9\},\{10\}\right\}$ is a $2$-partition basis of this $T$ and $|\Pi|=5=\dim_2(T)+1$.}
\label{exampleAchievesUpperBound}
\end{figure}

However, there are trees for which the bound \eqref{trivialUpperBoundTree} is not tight, for instance for any star graph $K_{1,n}$ such that $n\ge 3$, it follows that $\dim_2(K_{1,n})=\pd_2(K_{1,n})=n$. We now propose another upper bound that is less than or equal to given in \eqref{trivialUpperBoundTree}. To this end, we need to introduce some new definitions. We denote by 
$$\tau=\max_{w\in\mathcal{M}(T)}\{|U(w)|\}=\max_{w\in\mathcal{M}(T)}\{\ter(w)\}.$$
We also define a partition $\Pi=\{\mathcal{M}_1(T),\ldots,\mathcal{M}_r(T)\}$ of $\mathcal{M}(T)$ such that two vertices $w,w'\in\mathcal{M}_i(T)\in\Pi$ if $l(w)=l(w')=l_i$ and $l_1<l_2<\ldots<l_r$. A set $\mathcal{M}_i(T)\in\Pi$ is composed by only one vertex if and only if there exist only one vertex $w\in\mathcal{M}(T)$ such that $l(w)=l_i$. Given $k\in \{1,\ldots,\varsigma(T)\}$ we define $s_k$ as the maximum subscript such that $l_{s_k}\le\lfloor\frac{k}{2}\rfloor$ and if $l_1>\lfloor\frac{k}{2}\rfloor$, then we assume that $s_k=1$. Given $\mathcal{M}_i(T)\in\Pi$ we denote $\displaystyle t_i=\max_{w\in\mathcal{M}_i(T)}\{|U(w)|\}$. Finally, if we consider that $\displaystyle\max_{s_k<j\le r}\{t_j\}=0$ for $s_k=r$, then  we can define the following parameter: 
\begin{align*}
\mathcal{I}_k(T)=&(t_1-2)\max\left\{k-l_1,\left\lceil\frac{k}{2}\right\rceil\right\}+\sum_{i=2}^{s_k}\max\left\{t_i-\max_{1\le j<i}\{t_j\},0\right\}\left(k-l_i\right)+\\
&+\max\left\{\max_{s_k<j\le r}\{t_j\}-\max_{1\le j\le s_k}\{t_j\},0\right\}\left\lceil\frac{k}{2}\right\rceil.
\end{align*}
Figure~\ref{clarifyingExample} shows an example helping to clarify the parameter $\mathcal{I}_k(T)$. On the other hand, from Figure~\ref{exampleAchievesUpperBound} it follows that $\mathcal{M}_1(T)=\{3,8\}=\mathcal{M}(T)$, $l_1=1$, $t_1=2$, and as a consequence, $s_2=r=1$ and $\mathcal{I}_2(T)=0$.

\begin{figure}[!ht]
\centering
\begin{tikzpicture}[transform shape, inner sep = .8mm]
\node [draw=black, shape=circle, fill=white] (2) at (0,0) {};
\node [scale=.8] at ([yshift=-.5 cm]2) {$2$};
\node [draw=black, shape=circle, fill=white] (3) at ([xshift=-2 cm]2) {};
\node [scale=.8] at ([yshift=-.5 cm]3) {$3$};
\node [draw=black, shape=circle, fill=white] (1) at ([xshift=2 cm]2) {};
\draw (1)--(2)--(3);
\node [scale=.8] at ([yshift=-.5 cm]1) {$1$};
\node [draw=black, shape=circle, fill=white] (4) at ([shift=(45:sqrt(2)cm)]1) {};
\node [draw=black, shape=circle, fill=white] (5) at ([xshift=1 cm]1) {};
\foreach \x in {6,...,9}
{
\pgfmathtruncatemacro\predecessor{\x-1};
\node [draw=black, shape=circle, fill=white] (\x) at ([xshift=1 cm]\predecessor) {};
\draw (\predecessor)--(\x);
}
\node [draw=black, shape=circle, fill=white] (10) at ([shift=(-45:sqrt(2)cm)]1) {};
\foreach \x in {11,...,14}
{
\pgfmathtruncatemacro\predecessor{\x-1};
\node [draw=black, shape=circle, fill=white] (\x) at ([xshift=1 cm]\predecessor) {};
\draw (\predecessor)--(\x);
}
\foreach \x in {4,5,10}
{
\draw (1)--(\x);
}
\node [draw=black, shape=circle, fill=white] (15) at ([shift=(36:1.5cm)]2) {};
\foreach \x in {16,...,18}
{
\pgfmathtruncatemacro\predecessor{\x-1};
\node [draw=black, shape=circle, fill=white] (\x) at ([yshift=1 cm]\predecessor) {};
\draw (\predecessor)--(\x);
}
\node [draw=black, shape=circle, fill=white] (19) at ([shift=(72:.94cm)]2) {};
\foreach \x in {20,...,22}
{
\pgfmathtruncatemacro\predecessor{\x-1};
\node [draw=black, shape=circle, fill=white] (\x) at ([yshift=1 cm]\predecessor) {};
\draw (\predecessor)--(\x);
}
\node [draw=black, shape=circle, fill=white] (23) at ([shift=(108:.94cm)]2) {};
\foreach \x in {24,...,26}
{
\pgfmathtruncatemacro\predecessor{\x-1};
\node [draw=black, shape=circle, fill=white] (\x) at ([yshift=1 cm]\predecessor) {};
\draw (\predecessor)--(\x);
}
\node [draw=black, shape=circle, fill=white] (27) at ([shift=(144:1.5cm)]2) {};
\node [draw=black, shape=circle, fill=white] (28) at ([yshift=1 cm]27) {};
\draw (27)--(28);
\foreach \x in {15,19,23,27}
{
\draw (2)--(\x);
}
\node [draw=black, shape=circle, fill=white] (29) at ([shift=(120:1.93cm)]3) {};
\foreach \x in {30,31}
{
\pgfmathtruncatemacro\predecessor{\x-1};
\node [draw=black, shape=circle, fill=white] (\x) at ([xshift=-1 cm]\predecessor) {};
\draw (\predecessor)--(\x);
}
\node [draw=black, shape=circle, fill=white] (32) at ([shift=(150:1.15cm)]3) {};
\foreach \x in {33,34}
{
\pgfmathtruncatemacro\predecessor{\x-1};
\node [draw=black, shape=circle, fill=white] (\x) at ([xshift=-1 cm]\predecessor) {};
\draw (\predecessor)--(\x);
}
\node [draw=black, shape=circle, fill=white] (35) at ([shift=(180:1cm)]3) {};
\foreach \x in {36,37}
{
\pgfmathtruncatemacro\predecessor{\x-1};
\node [draw=black, shape=circle, fill=white] (\x) at ([xshift=-1 cm]\predecessor) {};
\draw (\predecessor)--(\x);
}
\node [draw=black, shape=circle, fill=white] (38) at ([shift=(210:1.15cm)]3) {};
\foreach \x in {39,40}
{
\pgfmathtruncatemacro\predecessor{\x-1};
\node [draw=black, shape=circle, fill=white] (\x) at ([xshift=-1 cm]\predecessor) {};
\draw (\predecessor)--(\x);
}
\node [draw=black, shape=circle, fill=white] (41) at ([shift=(240:1.93cm)]3) {};
\foreach \x in {42,43}
{
\pgfmathtruncatemacro\predecessor{\x-1};
\node [draw=black, shape=circle, fill=white] (\x) at ([xshift=-1 cm]\predecessor) {};
\draw (\predecessor)--(\x);
}
\foreach \x in {29,32,35,38,41}
{
\draw (3)--(\x);
}
\end{tikzpicture}
\caption{In this tree $T$ we have that $\mathcal{M}_i(T)=\{i\}$ for $i\in\{1,2,3\}$, $l_1=1$, $l_2=2$, $l_3=3$, $t_1=3$, $t_2=4$, $t_3=5$. For $k=6$ it follows that $s_6=3$ and $\mathcal{I}_6(T)=12$.}
\label{clarifyingExample}
\end{figure}
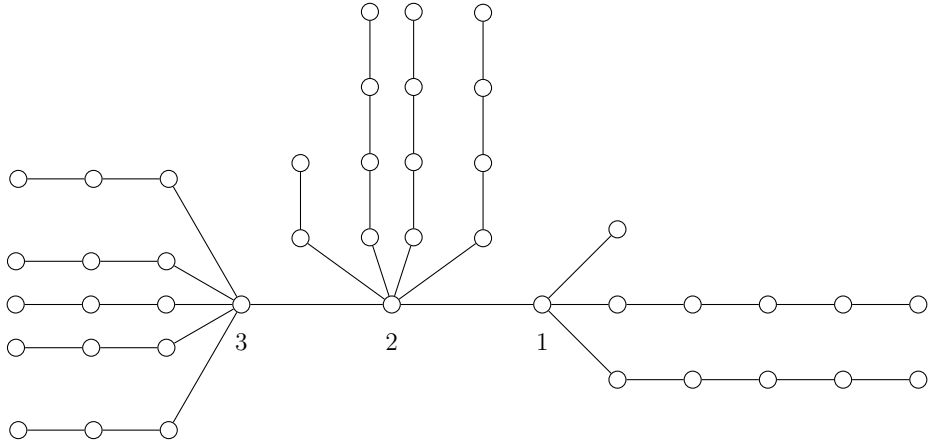

With this notation in mind we can present the next result.

\begin{theorem}\label{theoUpperBoundTree}
Let $T=(V,E)$ be a tree which is not a path. For $k\in \{1,\ldots, \varsigma(T)\}$ we have that, 
$$\pd_k(T)\le k\kappa+\mathcal{I}_k(T)+1,$$
where $\kappa=|\mathcal{M}(T)|$.
\end{theorem}

\begin{proof}
If $k=1$, then $s_1=1$ and we deduce that $\mathcal{I}_1(T)=\tau - 2$, and as a consequence, Theorem~\ref{theoUpperBoundTree} leads to $\pd_1\le\kappa+\tau-1$. This result was previously obtained in \cite{RodriguezVelazquez2014}. Hence from now we consider that $k\ge 2$.

Since $T$ is not a path, $T$ contains at least one vertex $w_i$ belonging to $\mathcal{M}(T)=\{w_1,\ldots,w_{\kappa}\}$. We suppose $u_{1}^i$ is a terminal vertex of $w_i\in\mathcal{M}(T)$ such that $l(u_{1}^i,w_i)=l(w_i)$. Let $U(w_i)=\{u_1^i,u_{2}^i,\ldots,u_{r_i}^i\}$ be the set of terminal vertices of $w_i$. Now, for every $u_{j}^i\in U(w_i)$, let the path $P(w_i,u_{j}^i)=w_iv^1_{j}v^2_{j}\cdots v_{j}^{d(w_i,u_j^i)-1}u_{j}^i$ and we consider the set $S_{ij}=V\left(P(w_i,u_{j}^i)\right)-\{w_i\}$. Given $k\in \{1,\ldots, \varsigma(T)\}$, we define the partition $A_{ij}$ on $S_{ij}$ in the following way. If $l(w_i)\le\lfloor\frac{k}{2}\rfloor$, then $A_{i1}=\left\{A_{i1}^1,\ldots,A_{i1}^{l(w_i)}\right\}$ and $A_{ij}=\left\{A_{ij}^1,\ldots,A_{ij}^{k-l(w_i)}\right\}$, where $j\ne 1$, $A_{i1}^l=\left\{v^l_{1}\right\}$ ($1\le l\le l(w_i)-1$), $A_{i1}^{l(w_i)}=\left\{u_1^i\right\}$, $A_{ij}^l=\left\{v^l_{j}\right\}$ ($1\le l\le k-l(w_i)-1$) and $A_{ij}^{k-l(w_i)}=S_{ij}-\bigcup_{l=1}^{k-l(w_i)-1}v_j^l$. On the other hand, if $l(w_i)>\lfloor\frac{k}{2}\rfloor$, then $A_{i1}=\left\{A_{i1}^1,\ldots,A_{i1}^{\lfloor\frac{k}{2}\rfloor}\right\}$ and $A_{ij}=\left\{A_{ij}^1,\ldots,A_{ij}^{\lceil\frac{k}{2}\rceil}\right\}$, where $j\ne 1$, $A_{i1}^l=\left\{v^l_{1}\right\}$ ($1\le l\le \lfloor\frac{k}{2}\rfloor-1$), $A_{i1}^{\lfloor\frac{k}{2}\rfloor}=S_{i1}-\bigcup_{l=1}^{\lfloor\frac{k}{2}\rfloor-1}v_1^l$, $A_{ij}^l=\left\{v^l_{j}\right\}$ ($1\le l\le\lceil\frac{k}{2}\rceil-1$) and $A_{ij}^{\lceil\frac{k}{2}\rceil}=S_{ij}-\bigcup_{l=1}^{\lceil\frac{k}{2}\rceil-1}v_j^l$. In other words, each partition $A_{ij}$ on $S_{ij}$ it defines as follows

$$
A_{i1}=\left\{
\begin{array}{ll}
\left\{\left\{v^1_{1}\right\},\ldots,\left\{v^{l(w_i)-1}_{1}\right\},\left\{u_1^i\right\}\right\}, & \text{if } l(w_i)\le\lfloor\frac{k}{2}\rfloor 
\\
\\
\left\{\left\{v^1_{1}\right\},\ldots,\left\{v^{\lfloor\frac{k}{2}\rfloor-1}_{1}\right\},S_{i1}-\bigcup_{l=1}^{\lfloor\frac{k}{2}\rfloor-1}v_1^l\right\}, & \text{if } l(w_i)>\lfloor\frac{k}{2}\rfloor. 
\end{array}
\right.
$$
and for $j\ne 1$,
$$
A_{ij}=\left\{
\begin{array}{ll}
\left\{\left\{v^1_{j}\right\},\ldots,\left\{v^{k-l(w_i)-1}_{j}\right\},S_{ij}-\bigcup_{l=1}^{k-l(w_i)-1}v_j^l\right\}, & \text{if } l(w_i)\le\lfloor\frac{k}{2}\rfloor,\\
\\
\left\{\left\{v^1_{j}\right\},\ldots,\left\{v^{\lceil\frac{k}{2}\rceil-1}_{j}\right\}, S_{ij}-\bigcup_{l=1}^{\lceil\frac{k}{2}\rceil-1}v_j^l\right\}, & \text{if } l(w_i)>\lfloor\frac{k}{2}\rfloor.
\end{array}
\right.
$$
We consider that $A_{ij}^l=\emptyset$ for either $j>|U(w_i)|=\ter(w_i)$ or $1<j\le |U(w_i)|=\ter(w_i)$ and $l>\max\left\{k-l(w_i),\left\lceil\frac{k}{2}\right\rceil\right\}$. Let us show that
$$
\Pi=\left\{
\begin{array}{ll}
A\cup C&\text{if } \tau=2\\
A\cup B\cup C&\text{otherwise,}
\end{array}
\right.
$$
is a $k$-partition generator of $T$, where 
$$A=\bigcup_{i=1}^\kappa\left(\bigcup_{l=1}^{\min\left\{l(w_i),\left\lfloor\frac{k}{2}\right\rfloor\right\}}\{A_{i1}^l\}\cup\bigcup_{l=1}^{\max\left\{k-l(w_i),\left\lceil\frac{k}{2}\right\rceil\right\}}\{A_{i2}^l\}\right),$$
$$B=\bigcup_{j=3}^\tau\bigcup_{l=1}^{\max_{1\le i\le\kappa}\left\{|A_{ij}|\right\}}\{B_j^l\},$$
$$B_j^l=\bigcup_{i=1}^\kappa A_{ij}^l,$$
and finally,
$$
C=\left\{
\begin{array}{ll}
V-\bigcup_{i=1}^\kappa\left(\bigcup_{l=1}^{\min\left\{l(w_i),\left\lfloor\frac{k}{2}\right\rfloor\right\}}A_{i1}^l\cup\bigcup_{l=1}^{\max\left\{k-l(w_i),\left\lceil\frac{k}{2}\right\rceil\right\}}A_{i2}^l\right)&\text{if } \tau=2,\\
&\\
V-\bigcup_{i=1}^\kappa\left(\bigcup_{l=1}^{\min\left\{l(w_i),\left\lfloor\frac{k}{2}\right\rfloor\right\}}A_{i1}^l\cup\bigcup_{l=1}^{\max\left\{k-l(w_i),\left\lceil\frac{k}{2}\right\rceil\right\}}A_{i2}^l\right)-&\text{otherwise.}\\
-\bigcup_{j=3}^\tau\bigcup_{l=1}^{\max_{1\le i\le\kappa}\left\{|A_{ij}|\right\}} B_j^l&
\end{array}
\right.
$$
It can be noted that $|A|=k\kappa$, and if $\tau\ge 3$, then $|B|=\mathcal{I}_k(T)$ and otherwise $\mathcal{I}_k(T)=0$. Thus, $|\Pi|=k\kappa+\mathcal{I}_k(T)+1$. We consider two different vertices $x,y\in V$.
\begin{enumerate}[{Case} 1:]
\item $x\in A_{ij}^l$ and $y\in A_{ij'}^{l'}$. If $w_i$ distinguishes $x,y$, then at least $k$ of the $k+1$ vertex sets in $A_{i1}\cup A_{i2}\cup C$ distinguish $x,y$. Suppose that $w_i$ does not distinguish $x,y$. In this case $j\ne j'$. If $j,j'\ge 3$, then $x,y$ are distinguished by all vertex sets in $\displaystyle\bigcup_{l=1}^{|A_{ij}|}\{B_j^l\}\cup \bigcup_{l=1}^{|A_{ij'}|}\{B_{j'}^l\}$. If $j<3$ and $j'\ge 3$, then all vertex sets in $\displaystyle A_{ij}\cup \bigcup_{l=1}^{|A_{ij'}|}\{B_{j'}^l\}$ distinguish $x,y$. By last, if $j,j'<3$, then $x,y$ are distinguished by all vertex sets in $A_{i1}\cup A_{i2}$. In any case, $x,y$ are distinguished by at least $k$ vertex sets of $\Pi$.
\item $x\in A_{ij}^l$, $y\in A_{i'j'}^{l'}$ and $i\ne i'$. In this case $w_i$ or $w_{i'}$ distinguish $x,y$, say $w_i$. Thus, $x,y$ are distinguished by the $k$ vertex sets in $A_{i1}\cup A_{i2}$.
\item $x\in A_{ij}^l$ and $y\in C$. If $|C|=1$, then $i=1$, $\{y\}=\{w_1\}=\mathcal{M}(T)$ and since $T$ is not a path, $y$ has at least three terminal vertex, \emph{i.e.}, $\tau=\ter(y)\ge 3$. In this case $x,y$ are distinguished by all vertex sets in $\displaystyle\bigcup_{r=1,r\ne j}^\tau A_{1r}\subseteq\Pi$. Suppose that $|C|\ge 2$. In this case there exists another vertex $w_j\in\mathcal{M}(T)$ different from $w_i$ such that $y$ belongs to the path from $w_j$ to $w_i$. Thus $x,y$ is distinguished by all vertex sets in $A_{j1}\cup A_{j2}$. In any case, $x,y$ are distinguished by at least $k$ vertex sets of $\Pi$.
\item $x,y\in C$. In this case there exists $w_i\in\mathcal{M}(T)$ such that $d(x,w_i)<d(y,w_i)$. Hence $x,y$ is distinguished by the $k$ vertex sets in $A_{i1}\cup A_{i2}$.
\end{enumerate}
Therefore, $\Pi$ is a $k$-partition generator of $T$, and as a consequence, $\pd_k(T)\le |\Pi|=k\kappa+\mathcal{I}_k(T)+1$.
\end{proof}

Two interesting particular cases are the following ones for $k=2$ and $k=3$, respectively. For these two cases $\mathcal{I}_2(T)=\tau - 2$ and $\mathcal{I}_3(T)=2\tau - 4$. Therefore, by Theorem~\ref{theoUpperBoundTree} we have the next results.

\begin{corollary}
If $T$ is a tree different from a path, then
$$\pd_2(T)\le 2\kappa+\tau-1,$$
where $\kappa=|\mathcal{M}(T)|$.
\end{corollary}

\begin{corollary}
If $T$ is a tree different from a path such that $ \varsigma(T)\ge 3$, then
$$\pd_3(T)\le 3\kappa+2\tau-3,$$
where $\kappa=|\mathcal{M}(T)|$.
\end{corollary}

For tree $T$ shown in Figure~\ref{clarifyingExample} the upper bound  on $\pd_6(T)$ given by Theorem~\ref{theoUpperBoundTree} is notably better than upper bound \eqref{trivialUpperBoundTree}. In this case $\pd_6(T)\le 3\cdot 6+12+1=31<41=11+14+15+1=\dim_6(T)+1$. However, if we consider the tree $T$ depicted in Figure~\ref{exampleAchievesUpperBound}, then these two upper bounds on $\pd_2(T)$ are the same and equal to $2$-metric dimension of $T$ increased in one.

\section*{Acknowledgement.}
We are very grateful to the anonymous referees for the evaluation of our paper and for the constructive criticism. This work has been partially supported by the ``Ministerio de Econom{\'i}a y Competitividad" (TRA2013-48180-C3-P and TRA2015-71883-REDT).

\end{document}